\documentclass[11pt]{article}

\newcommand\version{September 20, 2026}

\usepackage[margin=1in]{geometry}
\usepackage[T1]{fontenc}
\usepackage{lmodern}
\usepackage{amsmath,amssymb,amsthm,mathtools}
\usepackage{microtype}
\usepackage{xcolor}
\usepackage[colorlinks=true,linkcolor=blue!50!black,citecolor=blue!50!black,urlcolor=blue!50!black]{hyperref}
\hypersetup{pdftitle={Sharp comparison between the perimeter and its fractional analogue in two dimensions},pdfauthor={}}
\newtheorem{theorem}{Theorem}[section]
\newtheorem{lemma}[theorem]{Lemma}
\newtheorem{proposition}[theorem]{Proposition}

\theoremstyle{remark}

\newcommand{\R}{\mathbb R}
\newcommand{\dd}{\,\mathrm d}

\numberwithin{equation}{section}
\allowdisplaybreaks
\begin{document}

\title{Sharp comparison between the perimeter and its fractional analogue in two dimensions}

\author{Rupert L. Frank\thanks{Mathematisches Institut, Ludwig-Maximilians Universit\"at M\"unchen, Theresienstr.~39, 80333 M\"unchen, Germany; and Munich Center for Quantum Science and Technology, Schellingstr.~4, 80799 M\"unchen, Germany. Email: \texttt{r.frank@lmu.de}.}
\and
Paata Ivanisvili\thanks{Department of Mathematics, University of California, Irvine, 510C Rowland Hall, Irvine, CA 92697-3875, USA. Email: \texttt{pivanisv@uci.edu}.}}

\date{\version}

\maketitle

\begin{abstract}
We show that the fractional perimeter of any order $0<s<1$ is maximal, among all planar sets of fixed perimeter, for the disk.
\end{abstract}


\section{Main result}

We are interested in measurable sets $E\subset\R^2$ whose perimeter in the sense of De Giorgi, denoted by $P(E)$, is finite. If such a set has finite measure, then the isoperimetric inequality states that
\begin{equation}
    \label{eq:isoper}
    P(E) \geq \sqrt{4\pi}\, |E|^{1/2}
\end{equation}
with equality if and only if $E$ is a disk, up to sets of measure zero. We refer to \cite{Maggi} for background on finite perimeter sets.

Of great interest recently has been the notion of fractional perimeter of order $s\in(0,1)$, which for a measurable set $E\subset\R^2$ is defined by
\begin{equation}\label{eq:defper}
 P_s(E) :=\iint_{E\times E^c} \frac{\dd y\dd x}{|x-y|^{2+s}} \,.
\end{equation}
The fractional isoperimetric inequality is due to \cite{AL} with cases of equality classified in \cite{FS}. It states that
\begin{equation}
    \label{eq:isoperfrac}
    P_s(E) \geq \frac{2^{2-s}\,\pi^{(1+s)/2}}{s(2-s)} \,
    \frac{\Gamma((1-s)/2)}{\Gamma(1-s/2)} \, |E|^{(2-s)/2}.
\end{equation}
with equality if and only if $E$ is a disk, up to sets of measure zero.

Our main result in this paper is a sharp inequality that, in view of \eqref{eq:isoperfrac}, strengthends the classical isoperimetric inequality \eqref{eq:isoper}.

\begin{theorem}\label{thm:plane}
For every $0<s<1$ and every measurable $E\subset\R^2$ with $0<|E|<\infty$ and $P(E)<\infty$, we have
\begin{equation}\label{eq:planar-main}
 P_s(E)\leq \frac{\pi^{s-1/2}}{s(2-s)}\ \frac{\Gamma((1-s)/2)}{\Gamma(1-s/2)} \, P(E)^{2-s} \,.
\end{equation}
Equality holds if and only if $E$ agrees almost everywhere with a disk.
\end{theorem}

This theorem settles the two-dimensional case of an open problem posed by Maz'ya \cite[Section 3.1]{Mazya}. The higher-dimensional case remains open. The local minimality of the disk (or a ball) is not explicitly state in \cite{ACMM}, but can be deduced from their arguments. An even stronger conjecture \cite{ACMM}, in any dimension $n\geq 2$, would be that $(P_s(E)/P_s(E^*))^{1/(n-s)}$ is nondecreasing with respect to $s\in (0,1)$. Here $E^*$ is the disk (or ball) of the same measure as $E$. This remains open as well.

What makes Theorem \ref{thm:plane} conceptually interesting and difficult to prove is that, among sets of given measure, \emph{both} the left and the right side of \eqref{eq:planar-main} are minimal for a disk. This means there is a competition between the two sides that is stronger than in the classical or the fractional isoperimetric inequalities \eqref{eq:isoper} and \eqref{eq:isoperfrac}. Thus arguments based on symmetrization techniques, which provide proofs of \eqref{eq:isoper} and \eqref{eq:isoperfrac}, seem to be of no use for the proof of Theorem \ref{thm:plane}. Problems with similar competitions appear in several applications; see, for example, \cite{F23} and the recent breakthrough \cite{ChGi}.

The remainder of this paper is devoted to the proof of Theorem \ref{thm:plane}. Our strategy consists in first proving the inequality for convex sets (Section \ref{sec:planar-convex}) and then to extend it to general sets (Section \ref{sec:planar-general}). In both cases we make use of planarity.

To prove the inequality in the convex case, we begin by following a classical route in the calculus of variations. First, we prove the existence of an optimzer, then we derive the corresponding Euler--Lagrange equation, from which we deduce regularity of optimizers. With these facts at our disposal, we can finally use some subtle, but elementary ad-hoc arguments to deduce from the Euler--Lagrange equation that an optimizer necessarily has constant curvature.

The extension from the convex to the general case uses some reflection and convexification arguments that are typical for two-dimensional isoperimetric problems. While straightforward in the smooth case, we need to overcome some technical difficulties when working with general sets of finite perimeter.

Finally, we would like to mention the preprint `A Sharp Planar Fractional Isoperimetric Inequality' by Xiaosheng Lin, Dachun Yang, Sibei Yang, Wen Yuan and Yangyang Zhang, where closely related results are obtained independently and simultaneously. The first version of our paper (which coincides with the present version except for this paragraph) was submitted to the arXiv on September 13 and appeared there on September 15. On September 17, the aforementioned preprint appeared on the arXiv, after having been submitted on August 13. The two approaches are substantially different and give different perspectives on the same problem. While both arguments involve a reduction to the convex setting, our proof is based on an analysis of the Euler–Lagrange equation, while theirs is based on the parallel flow. Our approach allows us to classify the cases of equality in the sharp inequality.


\section{The inequality for convex sets}\label{sec:planar-convex}

Our objective in this section is to prove the following result.

\begin{theorem}\label{prop:planar-convex}
Theorem~\ref{thm:plane} holds under the additional assumption that $E$ is convex.
\end{theorem}

Without loss of generality we may restrict our attention to \emph{convex bodies}, that is, to compact, convex sets with non-empty interior.


\subsection{Existence of a convex
maximizer}\label{existence-of-a-convex-maximizer}

In this subsection we prove that the optimization corresponding to the inequality in Theorem \ref{prop:planar-convex},
\begin{equation}\label{eq:optproblem}
    \sup \left\{ \frac{P_s(K)}{P(K)^{2-s}} :\ K\subset\R^2 \ \text{convex body} \right\},
\end{equation}
has an optimizer. 

For the fact that \eqref{eq:optproblem} is finite, even without the convexity constraint, see \cite{Mazya}. By scaling, it is enough to maximize $P_s(K)$ over convex bodies with $P(K) = 1$. Such bodies have diameter at most \(1/2\). Translate every body to meet the origin; they then lie in a fixed disk. The Blaschke selection theorem \cite[Theorem 1.8.7]{Schneider} gives a Hausdorff-convergent subsequence $(K_j)$, with limit a compact convex set \(K_*\).

The fact that Hausdorff convergence of bounded convex bodies gives \(L^1\) convergence, together with the lower semicontinuity of the perimeter  under $L^1$-convergence \cite[Proposition 12.15]{Maggi}, gives
\[
P(K_*) \leq \liminf_{j\to\infty} P(K_j) = 1 \,.
\]
In passing, we mention that, when $K_*$ has nonempty interior, using Cauchy's formula (see \eqref{eq:plane-2} below), we can even deduce continuity of the perimeter, but lower semicontinuity suffices for our purposes.

We also need continuity of the fractional perimeter along this sequence. The inequality in Lemma \ref{lem:interpol} below, applied to \(u=\chi_{K_j}-\chi_{K_*}\), the fact that \([\chi_E]_{W^{s,1}}=2P_s(E)\) and the reverse triangle inequality give
\[
2 | P_s(K_j) - P_s(K_*)| \leq C_s |K_j\Delta K_*|^{1-s} \left( P(K_j) + P(K_*) \right)^s \leq 2 C_s |K_j\Delta K_*|^{1-s} \,.
\]
The $L^1$-convergence of $(K_j)$ therefore implies \(P_s(K_j)\to P_s(K_*)\).

Since the supremum in \eqref{eq:optproblem} is positive and since $P_s(K_j)$ converges to this number, we deduce that $P_s(K_*)$ is positive. This, in turn, implies that $K_*$ has nonempty interior. In particular, $P(K_*)>0$.

By lower semicontinuity of the perimeter and continuity of the fractional perimeter we deduce that $K_*$ attains the supremum in \eqref{eq:optproblem}. This is what we wanted to show.

\medskip

The argument in this subsection required the following simple interpolation inequality.

\begin{lemma}\label{lem:interpol}
    Let $0<s<1$. Then there is a constant $C_s<\infty$ such that for any $u\in BV(\R^2)$, one has
    \begin{equation}
[u]_{W^{s,1}(\mathbb R^2)}\leq C_s\|u\|_1^{1-s}|Du|(\mathbb R^2)^s. \label{eq:plane-3}
\end{equation}
\end{lemma}

\begin{proof}
    A trivial inequality and the fundamental theorem of calculus give the bound
\(\|u(\cdot+z)-u\|_1\leq\min\{2\|u\|_1,|z||Du|(\R^2)\}\). Multiplying this bound by $|z|^{-2-s}$ and integrating with respect to $z\in\R^2$ yields \eqref{eq:plane-3}.
\end{proof}


\subsection{The Euler--Lagrange equation}\label{first-variation-input-and-its-exact-normalization}

In this subsection we derive the Euler--Lagrange equation for a solution of the optimization problem in \eqref{eq:optproblem}.

We need to introduce some notation for a convex body $K\subset\R^2$. Its support function is defined by
\[
h_K(u) := \sup_{x\in K} u\cdot x
\qquad\text{for all}\ u\in\mathbb S^1 \,.
\]
The normal $\nu_K(x)$ is well defined for $\mathcal H^1$-a.e. $x\in\partial K$ and the map $\partial K \to\mathbb S^1$, $x\mapsto\nu_K(x)$, known as the \emph{Gauss map}, is continuous on the set where it is defined; see, e.g., \cite[Section 2.2]{Schneider}. (Here $\mathcal H^1$ denotes 1-dimensional Hausdorff measure.) The \emph{radial function} of $K$ is defined by
\[
\rho_{K,x}(u):=\sup\{r\geq0:x+ru\in K\}
\]
and the \emph{fractional mean curvature} of order $s\in(0,1)$ is defined by
\begin{align}
    \label{eq:meancurvfrac}
     H_s^K(x):=\frac2s\int_{\{u\in \mathbb S^1:u\cdot\nu_K(x)<0\}}
                 \rho_{K,x}(u)^{-s}\dd \mathcal H^1(u)
\end{align}
for every regular boundary point $x$ (that is, for every $x\in\partial K$ where $\nu_K$ is defined).

\begin{proposition}\label{eleq}
    Let $K$ be a convex body attaining the supremum in \eqref{eq:optproblem}. Then, as finite Borel measures on $\mathbb S^1$,
    \begin{equation}
        (\nu_K)_\#(H_s^K\dd\mathcal H^1|_{\partial K})=\lambda\dd\mathcal H^1|_{\mathbb S^1} \label{eq:plane-6}
    \end{equation}
    with
    \[
    \lambda=\frac{(2-s)P_s(K)}{P(K)}>0 \,.
    \]
\end{proposition}

We recall that \eqref{eq:plane-6} means that, for any $f\in C(\mathbb S^1)$,
    \begin{align}
        \label{eq:plane-6alt}
            \int_{\partial K}f(\nu_K(x))H_s^K(x)\dd\mathcal H^1(x)
       =\lambda\int_{\mathbb S^1}f(u)\dd\mathcal H^1(u) \,.
    \end{align}

\begin{proof}
    Let $K$ be as in the proposition. After a translation we may assume that the origin is in its interior. Then $h_K$ is strictly positive.

    For every positive function \(h\in C(\mathbb S^1)\), the set
    \begin{equation}
        \label{eq:wulff}
        [h] := \{ x:\ x\cdot u\leq h(u)\text{ for every }u\in S^1\}
    \end{equation}
    is a convex body containing the origin in its interior, and its support function satisfies \(h_{[h]}\leq h\). Therefore the planar Cauchy formula \cite[Eq.~(5.73)]{Schneider}
    \begin{equation}
        P(K')=\int_{\mathbb S^1} h_{K'}(u)\dd\mathcal H^1(u) \,, \label{eq:plane-2}
    \end{equation}
    valid for all convex bodies $K'\subset\R^2$, implies
    \[
        P([h])\leq\int_{\mathbb S^1} h(u)\dd\mathcal H^1(u) \,.
    \] 
    Maximality of \(K\) therefore gives 
    \begin{equation}
    P_s([h])\leq \frac{P_s(K)}{P(K)^{2-s}} \, P([h])^{2-s}
    \leq \frac{P_s(K)}{P(K)^{2-s}} \, \left(\int_{\mathbb S^1} h(u)\dd\mathcal H^1(u) \right)^{2-s}, \label{eq:plane-5}
    \end{equation}
    with equality when \(h=h_K\).

    We take \(h=h_K+tf\) with arbitrary \(f\in C(S^1)\) and with $t\in\R$ sufficiently close to zero. Differentiating \eqref{eq:plane-5} at its equality point $t=0$, gives 
    \[
    \frac{d}{dt}\Big|_{t=0} P_s([h_K+t f])
       =\lambda\int_{\mathbb S^1} f(u)\dd\mathcal H^1(u) \,,
    \]
    provided the derivative on the left side exists. That it does and the value of this derivative are provided in Lemma \ref{lem:firstvar} below. This leads to the claimed formula \eqref{eq:plane-6alt}.
\end{proof}

In the previous proof we used the following formula for the derivative of the fractional perimeter, which we will deduce from \cite[Theorem 5.5]{LXYZ}. We recall the notation \eqref{eq:wulff}.

\begin{lemma}\label{lem:firstvar}
    Let $K$ be a convex body and $s\in(0,1)$. Then $H_s^K\in L^1(\partial K,\dd\mathcal H^1)$. Moreover, for any $f\in C(\mathbb S^1)$, the function $t\mapsto P_s([h_K+t f])$ is differentiable at $t=0$ and
    \begin{align}
        \frac{d}{dt}\Big|_{t=0} P_s([h_K+t f]) = \int_{\partial K}f(\nu_K(x))H_s^K(x)\dd\mathcal H^1(x). \label{eq:plane-1}    
    \end{align}
\end{lemma}

\begin{proof}
    We present the proof in $n$ dimensions. Let $\omega_n$ denote the measure of the unit ball in $\R^n$. Following \cite{LXYZ}, for $q>0$ and a convex body $K\subset\R^n$, we set
    \[
    I_q(K)=\frac1{n\omega_n}\int_{\mathbb S^{n-1}}\int_{u^\perp}
                  |K\cap(y+\mathbb Ru)|^q\dd y\dd\mathcal H^{n-1}(u)
    \] 
    and, for a regular boundary point $x$ of $K$, 
    \[
    \widetilde V_{q-1}(K,x)=\frac1n\int_{\rho_{K,x}>0}\rho_{K,x}(u)^{q-1}\dd\mathcal H^{n-1}(u) \,.
    \] 
    According to \cite[Lemma 4.9]{LXYZ}, $\widetilde V_{q-1}(K,\cdot)\in L^1(\partial K,\mathcal H^{n-1})$ and, according to \cite[Theorem 5.5]{LXYZ}, $t\mapsto I_q([h_K+tf])$ is differentiable at $t=0$ with
    \begin{equation}
        \label{eq:lxyz}
            \frac{d}{dt}\Big|_{t=0} I_q([h_K+tf]) = \frac{2q}{\omega_n} \int_{\partial K} \widetilde V_{q-1}(K,x) f(\nu_K(x)) \dd\mathcal H^{n-1}(x) \,.
    \end{equation}

    In the remainder of the proof, we will show that
    \begin{equation}
        \label{eq:slicing}
        \int_K \int_{K^c} \frac{\dd x\dd y}{|x-y|^{n+s}}=\frac{n\omega_n}{s(1-s)}I_{1-s}(K).
    \end{equation}
    Thus, \eqref{eq:lxyz} with $q=1-s$ proves the lemma.

    To prove \eqref{eq:slicing}, we change variables $x=X+r/2$, $y=X-r/2$, and then we write $r= tu$ with $t\in(0,\infty)$ and $u\in\mathbb S^{n-1}$. Finally, for fixed $u\in \mathbb S^{n-1}$, we decompose $X=y+\lambda u$ with $y\in u^\bot$ and $\lambda\in\R$. In this way, we obtain
    \begin{align*}
        \int_K \int_{K^c} \frac{\dd x\dd y}{|x-y|^{n+s}} & = \int_{\R^n} \int_{\R^n} \mathbf1_{X+\tfrac 12 r \in K} \mathbf1_{X-\tfrac 12r\not\in K} \frac{\dd X\dd r}{|r|^{n+s}} \\
        & = \int_{\mathbb S^{n-1}} \int_0^\infty \int_{\R^n} \mathbf1_{X+\tfrac t2 u \in K} \mathbf1_{X-\tfrac t2u\not\in K} \dd X \, \frac{\dd t}{t^{1+s}} \, \dd\mathcal H^{n-1}(u) \\
        & = \int_{\mathbb S^{n-1}} \int_{u^\bot} \int_0^\infty \int_{\R} \mathbf1_{y+(\lambda +\tfrac t2) u \in K} \mathbf1_{y+(\lambda -\tfrac t2) u\not\in K} \dd\lambda \, \frac{\dd t}{t^{1+s}} \, \dd y \dd\mathcal H^{n-1}(u) \,.    
    \end{align*}
    Therefore, the claimed equality \eqref{eq:slicing} will follow, if we can show that for any $u\in\mathbb S^{n-1}$ and $y\in u^\bot$, we have
    \begin{align}
        \label{eq:slicing2}
        \int_0^\infty \int_{\R} \mathbf1_{y+(\lambda +\tfrac t2) u \in K} \mathbf1_{y+(\lambda -\tfrac t2) u\not\in K} \dd\lambda \, \frac{\dd t}{t^{1+s}}
        = \frac{1}{s(1-s)} |K\cap(y+\mathbb Ru)|^{1-s} \,.
    \end{align}
    In fact, we will show that for any $u\in\mathbb S^{n-1}$, $y\in u^\bot$ and $t>0$, we have
    \begin{align}
        \label{eq:slicing3}
        \int_{\R} \mathbf1_{y+(\lambda +\tfrac t2) u \in K} \mathbf1_{y+(\lambda -\tfrac t2) u\not\in K} \dd\lambda
        = \min\{ t, |K\cap(y+\mathbb Ru)| \} \,.
    \end{align}
    Indeed, \eqref{eq:slicing2} follows from \eqref{eq:slicing3} by integration with respect to $t$.

    To prove \eqref{eq:slicing3}, we note that there is a compact, possibly empty, interval $I\subset\R$ (depending on $u,y$) such that
    \[
    K\cap(y+\mathbb Ru) = \{ y+ \mu u :\ \mu\in I \} \,.
    \]
    Thus, \eqref{eq:slicing3} is equivalent to the one-dimensional identity
    \[
    \int_\R \mathbf1_{\lambda +\tfrac t2 \in I} \mathbf1_{\lambda - \tfrac t2 \not\in I} \dd\lambda = \min\{ t,|I|\} \,,
    \]
    which is elementary to check. This completes the proof of Lemma \ref{lem:firstvar}.
\end{proof}


\subsection{Regularity of convex critical points}\label{equation-6-excludes-corners-and-line-segments}

In this and the next subsection we turn our attention to solutions of the Euler--Lagrange equation \eqref{eq:plane-6} that we derived in the previous subsection. Optimality will not play a role anymore.

\begin{proposition}\label{prop:regularity}
    Let $K\subset\R^2$ be a convex body satisfying \eqref{eq:plane-6} for some $\lambda>0$. Then $\partial K$ is $C^2$ with positive curvature $\kappa$ and we have
    \begin{equation}
        \label{eq:eleqsmooth}
            H^K_s(x) = \lambda \, \kappa(x)
            \qquad\text{for all}\ x\in\partial K \,.
    \end{equation}
\end{proposition}

In the proof that follows we will use repeatedly that for any convex body $K\subset\R^2$ and any regular boundary point $x$, we have
\begin{equation}
    \label{eq:lowerbound}
    H_s^K(x) \geq (2\pi/s) D^{-s}
    \qquad\text{with}\ D:=\operatorname{diam} K \,.
\end{equation}
Indeed, note that $\rho_{K,x}(u) \leq D$ for all regular boundary points $x$ and all $u\in\mathbb S^1$ with $u\cdot\nu_K(x)<0$ and note that for a regular boundary point $x$ the set of $u\in\mathbb S^1$ with $u\cdot\nu_K(x)<0$ has measure $\pi$. Therefore, \eqref{eq:lowerbound} follows directly from the definition \eqref{eq:meancurvfrac} of the fractional mean curvature.

\begin{proof}
\emph{Step 1.} 
Suppose that \(K\) has a boundary point $x$ with more than one supporting tangent line. Then the interior of its normal cone $\{ u\in\mathbb S^1:\ u\cdot(y-x)\leq 0\  \text{for all}\ y\in K\}$ is a nonempty open arc, which we denote by $I$. For a positive function $f\in C(\mathbb S^1)$ supported in $I$, the left side of \eqref{eq:plane-6alt} vanishes, while the right side is positive. This is
impossible. Thus every boundary point has a unique supporting tangent normal. A
planar convex boundary with this property is \(C^1\).

Suppose instead that the boundary contains a nontrivial line segment.
Its relative interior has one fixed normal \(u_0\in\mathbb S^1\). In view of \eqref{eq:lowerbound}, the left side of \eqref{eq:plane-6} has a
positive atom at \(u_0\), contradicting the nonatomic right side.
Hence \(K\) is strictly convex.

\medskip

\emph{Step 2.}
As a consequence of what we have shown in Step 1, the Gauss map is a homeomorphism between \(\partial K\) and the circle $\mathbb S^1$. Let \(\gamma:\mathbb R/L\mathbb Z\to\partial K\) be
a positively oriented arclength parametrization of $\partial K$, with \(L=P(K)\), and let $\vartheta:\R \to \R$ be continuous, increasing function satisfying
\[
\gamma'(a) = (-\sin\vartheta(a), \cos\vartheta(a))^{\rm T}
\qquad\text{for all}\ a\in\R \,.
\]
(The continuity and monotonicity of $\vartheta$ come from the $C^1$-regularity and strict convexity of $K$, respectively.) Note that $\vartheta$ is not periodic, but satisfies $\vartheta(a+L) = \vartheta(a) + 2\pi$ for all $a\in\R$. This parametrization is chosen such that the normal satisfies
\begin{equation}
    \label{eq:normal}
    \widetilde\nu(a) := \nu_K(\gamma(a)) = (\cos\vartheta(a),\sin\vartheta(a))^{\rm T}
    \qquad\text{for all}\ a\in\R \,.
\end{equation}
Moreover, we write
\[
 \widetilde H_s^K(a) :=H_s^K(\gamma(a))
 \qquad\text{for all}\ a\in\R/L\mathbb Z
\]
and claim that the Euler--Lagrange equation \eqref{eq:plane-6} is equivalent to the equation, as measures on $\R/L\mathbb Z$,
\begin{equation}
\widetilde H_s^K(a)\dd a = \lambda\dd\vartheta(a) \,. \label{eq:plane-7}
\end{equation} 

Indeed, let $g\in C(\R/L\mathbb Z)$ and define $f\in C(\mathbb S^1)$ by
\[
f(\nu_K(\gamma(a))) = g(a)
\qquad\text{for all}\ a\in\R/L\mathbb Z \,.
\]
This is well defined by the homeomorphism properties of $\nu_K$ and $\gamma$, and we have
\[
\int_{\R/L\mathbb Z} g(a) \widetilde H_s^K(a)\dd a = \int_{\partial K} f(\nu_K(x)) H_s^K(x) \dd\mathcal H^1(x) \,.
\]
Meanwhile, in view of \eqref{eq:normal}, a change of variables gives
\[
\int_{\R/L\mathbb Z} g(a) \dd \vartheta(a) = \int_{\mathbb S^1} f(u)\dd\mathcal H^1(u) \,.
\]
The previous two equations, together with \eqref{eq:plane-6}, give
\[
\int_{\R/L\mathbb Z} g(a) \widetilde H_s^K(a)\dd a = \lambda \int_{\R/L\mathbb Z} g(a) \dd \vartheta(a) \,.
\]
This proves \eqref{eq:plane-7}.

Since $H_s^K \in L^1(\partial K,\mathcal H^1)$ by Lemma \ref{lem:firstvar}, we deduce from \eqref{eq:plane-7} that \(\vartheta\) is absolutely continuous and 
\begin{equation}
\vartheta'=\frac{\widetilde H_s^K}\lambda\in L^1(\R/L\mathbb Z) \,. \label{eq:plane-8altb}
\end{equation}
Let $\kappa:\partial K \to \R$ be the curvature of $\partial K$ and 
\[
\widetilde\kappa(a) := \kappa(\gamma(a))
\qquad\text{for all}\ a\in\R/L\mathbb Z \,.
\]
Then $\widetilde\kappa = \vartheta'$ and therefore \eqref{eq:plane-8altb} implies that
\begin{equation}
\widetilde\kappa=\frac{\widetilde H_s^K}\lambda \quad\text{a.e.}
\label{eq:plane-8}
\end{equation}

\medskip

\emph{Step 3.}
Next, we prove higher regularity.

For a \(C^1\) convex curve $\partial K$, the formula \eqref{eq:meancurvfrac} for \(H_s^K\) can be rewritten as 
\begin{equation}
H_s^K(x)=\frac2s\int_{\partial K}
       \frac{(y-x)\cdot\nu_K(y)}{|y-x|^{2+s}}\dd\mathcal H^1(y). \label{eq:plane-9}
\end{equation} 
This formula is well known and appears, for instance, in \cite[Equation (6)]{CSV} with a short indication of a proof. Let us indicate how to derive it from \eqref{eq:meancurvfrac}. Indeed, the radial projection \(y\mapsto(y-x)/|y-x|\) has angular Jacobian \(((y-x)\cdot\nu_K(y))/|y-x|^2\); applying the one-dimensional
area formula to arcs away from \(x\) and then monotone convergence
proves \eqref{eq:plane-9}. Formula \eqref{eq:plane-9} may initially be infinite on an exceptional null set, which causes no difficulty below.

Equivalently, \eqref{eq:plane-9} states that
\begin{equation}
\widetilde H_s^K(a)=\frac2s\int_{\R/L\mathbb Z}
       \frac{(\gamma(a+t)-\gamma(a))\cdot\widetilde\nu(a+t)}{|\gamma(a+t)-\gamma(a)|^{2+s}}\dd t \label{eq:plane-9alt} \,.
\end{equation}

To bound the denominator in \eqref{eq:plane-9alt} we use the fact that, by uniform continuity of the tangent of the compact \(C^1\) curve, there
exist \(0<r<L/2\) and \(c>0\) such that for all $a\in\R/L\mathbb Z$ one has
\begin{equation}
|\gamma(a+t)-\gamma(a)|\geq c|t|\quad (|t|\leq r), \label{eq:plane-10}
\end{equation} 

To bound the numerator, we write
\[
(\gamma(a+t)-\gamma(a))\cdot\widetilde\nu(a+t) =\int_0^t \gamma'(a+v)\cdot\widetilde\nu(a+t)\dd v\,.
\]
Using
\begin{align*}
    \gamma'(a+v)\cdot\widetilde\nu(a+t) & = -\sin\vartheta(a+v)\cos\vartheta(a+t) + \cos\vartheta(a+v) \sin\vartheta(a+t) \\
    & = \sin(\vartheta(a+t) - \vartheta(a+v)) \,,
\end{align*}
we find
\[
|\gamma'(a+v)\cdot\widetilde\nu(a+t)| \leq |\vartheta(a+t) - \vartheta(a+v)| = \left| \int_v^t \widetilde\kappa(a+w)\dd w \right|. 
\]
Consequently, for all $t\in\R$,
\begin{equation}
    \label{eq:plane-11}
    |(\gamma(a+t)-\gamma(a))\cdot\widetilde\nu(a+t)| \leq \left| \int_0^t \left| \int_v^t \widetilde\kappa(a+w)\dd w \right| \dd v \right| = \int_0^t w \widetilde\kappa(a+w)\dd w\,.
\end{equation}

We split the integration in \eqref{eq:plane-9alt} into the region $|t|\leq r$ and its complement. By \eqref{eq:plane-10} and \eqref{eq:plane-11}, the contribution of the former region is bounded by
\begin{align*}
    \frac 2s \, c^{-2-s} \int_{-r}^r |t|^{-2-s}\int_0^t w \widetilde\kappa(a+w) \dd w\dd t
    & = \frac 2s \, c^{-2-s} \int_{-r}^r |w| \widetilde\kappa(a+w) \int_{|w|}^r t^{-2-s}\dd t\dd w\\
    & \leq \frac 2s \, c^{-2-s} \frac1{1+s}\int_{-r}^r |w|^{-s}\widetilde\kappa(a+w)\,dw.
\end{align*}
For the region $r\leq |t|\leq L/2$, we bound
\[
\frac{(\gamma(a+t)-\gamma(a))\cdot\widetilde\nu(a+t)}{|\gamma(a+t)-\gamma(a)|^{2+s}} \leq \frac{1}{|\gamma(a+t)-\gamma(a)|^{1+s}} \leq C \,.
\]
Indeed, compactness and injectivity imply a positive lower bound on the Euclidean distance in this region. Therefore the contribution coming from this integral is uniformly bounded in $a$.

As a consequence, there is a constant $C'$ such that for almost every $a\in\R/L\mathbb Z$, we have
\begin{equation}
\widetilde H_s^K(a)\leq C' +C' \int_{-r}^r |w|^{-s}\widetilde\kappa(a+w)\,dw
       = C'+C'(k*\widetilde\kappa)(a), \label{eq:plane-12}
\end{equation}
where convolution is on the arclength circle and
\(k(w):=|w|^{-s}\mathbf1_{|w|<r}\). Thus, as a consequence of the Euler--Lagrange equation in the form \eqref{eq:plane-8}, we obtain
\begin{equation}
\widetilde H_s^K\leq C'+ \lambda^{-1} C'(k*\widetilde H_s^K) \,. \label{eq:plane-12alt}
\end{equation}

Based on this inequality, it is easy to complete the proof of regularity. We already know, by \eqref{eq:plane-8altb}, that \(\widetilde H_s^K\in L^1\). The kernel $k$ belongs to \(L^b\) for every
\(1\leq b<1/s\). Choose \(b>1\) in that range. Young's inequality and
\eqref{eq:plane-12alt} improve integrability from \(L^p\) to \(L^{p_1}\) whenever \[
 \frac1{p_1}=\frac1p+\frac1b-1>0.
\] 
Thus each application decreases \(1/p\) by the positive number
\(1-1/b\). By adjusting the exponent slightly if an endpoint is
encountered, finitely many iterations give \(\widetilde H_s^K \in L^p\) for some
\(p>1/(1-s)\). For this \(p\), its conjugate exponent satisfies
\(sp'<1\), so \(k\in L^{p'}\). Hölder's inequality in \eqref{eq:plane-12alt} then yields $\widetilde H_s^K\in L^\infty$. By the Euler--Lagrange equation in the form \eqref{eq:plane-8} we conclude
\begin{equation}
\widetilde\kappa\in L^\infty. \label{eq:plane-13}
\end{equation} 
The arclength tangent is therefore Lipschitz and the boundary is
\(C^{1,1}\).

For a \(C^{1,1}\) curve, the numerator in \eqref{eq:plane-9} is bounded by \(C|t|^2\)
on a local arc, and the chord bound \eqref{eq:plane-10} still holds. Hence the
contribution of \(|t|<\varepsilon\) to \eqref{eq:plane-9} is bounded, uniformly in the
basepoint $a$, by \(C\varepsilon^{1-s}\). The integral over the remaining
arcs is continuous in the basepoint. Uniform smallness of the local
contribution shows that \(\widetilde H_s^K\) is continuous everywhere on $\R/L\mathbb Z$. Equations \eqref{eq:plane-8} then identifies the almost-everywhere curvature with a continuous function and \eqref{eq:plane-8} holds everywhere. The curvature is positive by \eqref{eq:plane-8} and \eqref{eq:lowerbound}. Integrating the tangent equation
gives a \(C^2\) parametrization with continuous strictly positive
curvature, as claimed.
\end{proof}


\subsection{The rolling-disk argument}\label{an-elementary-rolling-disk-argument-forces-a-disk}

In this subsection we provide the essential ingredient for the proof of Theorem \ref{prop:planar-convex}.

\begin{proposition}\label{cor:uniform-chord}
    Let $K\subset\R^2$ be a $C^2$ convex body with positive curvature satisfying \eqref{eq:eleqsmooth} for some $\lambda>0$. Then $K$ is a disk.
\end{proposition}

\begin{proof}
    \emph{Step 1.}
    It is convenient to work with a different parametrization of $\partial K$ than in the previous subsection. We recall that $\nu_K$ is a homeomorphism from $\partial K$ to $\mathbb S^1$. Composing with the map $\iota:\R/2\pi\mathbb Z\to\mathbb S^1$, $\theta\mapsto (\cos\theta,\sin\theta)^{\rm T}$, we obtain the parametrization
    \[
    \alpha := \nu_K^{-1}\circ\iota : \R/2\pi\mathbb Z \to \partial K \,.
    \]
    This parametrization is distinguished by the fact that $\nu_K(\alpha(\theta))= (\cos\theta,\sin\theta)^{\rm T}$ for all $\theta\in\R/2\pi\mathbb Z$. We let
    \[
    \rho(\theta) := \frac{1}{\kappa(\alpha(\theta))}
    \qquad\text{for all}\ \theta\in\R/2\pi\mathbb Z
    \]
    denote the \emph{radius of curvature}. By assumption, $\rho$ is continuous and positive. In particular,
    \[
    \rho_{\rm min}:=\min\rho
    \qquad\text{and}\qquad \rho_{\rm max} :=\max\rho
    \]
    are both finite and positive.

    Let $h:=h_K\circ\iota$ be the support function of $K$ as a function of $\theta$. Then it is known that $h$ is $C^2$ and
    \[
    h'' + h = \rho
    \qquad\text{in}\ \R/2\pi\mathbb Z \,;
    \]
    see \cite[Section 2.5, in particular, Equation (2.60)]{Schneider}. Clearly, the supremum defining the support function is attained on the boundary of the convex body, so
    \[
    h(\theta) = \sup_{\theta'\in\R/2\pi\mathbb Z} \iota(\theta)\cdot\alpha(\theta') \,.
    \]
    It is known that this supremum is attained at $\theta'=\theta$, so
    \begin{equation}
        \label{eq:supportfcnmax}
        h(\theta) = \iota(\theta)\cdot\alpha(\theta)
        \qquad\text{for all}\ \theta\in\R/2\pi\mathbb Z\,.
    \end{equation}
    Differentiating this relation and using the fact that the tangent $\alpha'(\theta)$ is orthogonal to the normal $\iota(\theta)$, we get
    \begin{equation}
        \label{eq:supportfcnmax2}
        h'(\theta) = \iota'(\theta)\cdot\alpha(\theta)
        \qquad\text{for all}\ \theta\in\R/2\pi\mathbb Z\,.
    \end{equation}

    \medskip

    \emph{Step 2.}
    Consider a point $\alpha(\theta_0)\in\partial K$ and a disk of radius $r$ that is tangent to $\partial K$ at $\alpha(\theta_0)$. We parametrize the boundary of this disk by the normal angle, similarly as we parametrized $\partial K$, and we let $g:\R/2\pi\mathbb Z\to\R$ denote the support function of the disk in this parametrization. Then
    \[
    g(\theta) = r + x_0 \cdot \iota(\theta) \,,
    \]
    where $x_0\in\R^2$ is the center of the disk. The fact that the disk is tangent to $\partial K$ at $\alpha(\theta_0)$ implies that $x_0 = \alpha(\theta_0)-r \iota(\theta_0)$. This formula for $x_0$ implies that $g(\theta_0) = \alpha(\theta_0)\cdot\iota(\theta_0)$, which, in view of \eqref{eq:supportfcnmax}, implies that
    \begin{equation}
        \label{eq:gtheta0}
        g(\theta_0) = h(\theta_0) \,.
    \end{equation}
    Differentiating the formula for $g$ and using again the value of $x_0$ implies that $g'(\theta_0)= \iota'(\theta_0)\cdot\alpha(\theta_0)$, which, in view of \eqref{eq:supportfcnmax2}, implies that    
    \begin{equation}
        \label{eq:gtheta1}
        g'(\theta_0) = h'(\theta_0) \,.
    \end{equation}
    Finally, one easily verifies that $g'' + g = r$ in $\R/2\pi\mathbb Z$.

    Thus, the function 
    \[
    F:=h-g
    \]
    satisfies
    \[
    F''+F=\rho-r \,,\qquad F(\theta_0)=F'(\theta_0)=0,
    \] 
    so 
    \begin{equation}
    F(\theta_0+t)=\int_0^t\sin(t-v)[\rho(\theta_0+v)-r]\,dv
    \qquad\text{for all}\ t\in\R/2\pi\mathbb Z \,. \label{eq:plane-15}
    \end{equation} 
    For \(0\leq t\leq\pi\), the sine factor in the integrand is nonnegative. For \(-\pi\leq t\leq0\), the sine factor is nonpositive, but the integration orientiation reverses.

    Let us apply this with $r=\rho_{\rm min}$ and with $\theta_0 = \theta_{\rm min}$ such that $\rho(\theta_{\rm min}) = \rho_{\rm min}$. In this case the factor $[\rho(\theta_0+v)-r]$ in \eqref{eq:plane-15} is nonnegative and we conclude that $F(\theta_{\rm min}+t)\geq 0$ for all $|t|\leq\pi$. Thus, $h\geq g$ on $\R/2\pi\mathbb Z$ and we conclude that the disk of radius $\rho_{\rm min}$ tangent at $\alpha(\theta_{\rm min})$ is contained in $K$.

    Similarly, applying the argument with $r=\rho_{\rm max}$ and with $\theta_0=\theta_{\rm max}$ such that $\rho(\theta_{\rm max}) = \rho_{\rm max}$, and noting that the factor $[\rho(\theta_0+v)-r]$ in \eqref{eq:plane-15} is nonpositive and we conclude that $F(\theta_0+t)\leq 0$ for all $|t|\leq\pi$. Thus, $h\leq g$ on $\R/2\pi\mathbb Z$ and we conclude that the disk of radius $\rho_{\rm max}$ tangent at $\alpha(\theta_{\rm max})$ contains $K$.

Let \(c_s:=H_s^D\) denote the constant fractional mean curvature of the unit
disk. Then, by scaling, a radius-\(r\) disk has fractional curvature \(c_s r^{-s}\). At a
tangency point, inclusion implies the corresponding monotonicity of
fractional curvature: if two convex sets \(A\subset B\) are tangent at a common regular boundary point $x$, then \(H_s^B(x)\leq H_s^A(x)\). This follows immediately from \eqref{eq:meancurvfrac}.

Therefore, the containment of the minimal and maximal curvature disks that we proved above implies that
\[
H_s^K(\alpha(\theta_{\rm min})) \leq c_s \rho_{\rm min}^{-s}
\qquad\text{and}\qquad
H_s^K(\alpha(\theta_{\rm max})) \geq c_s \rho_{\rm max}^{-s} \,.
\]
At the same time, we have
\[
\kappa(\alpha(\theta_{\rm min})) = \frac{1}{\rho(\theta_{\rm min})} = \frac{1}{\rho_{\rm min}}
\qquad\text{and}\qquad
\kappa(\alpha(\theta_{\rm max})) = \frac{1}{\rho(\theta_{\rm max})} = \frac{1}{\rho_{\rm max}} \,.
\]
Combining these relations with the Euler--Lagrange equation \eqref{eq:eleqsmooth}, we obtain
\[
\frac{1}{\rho_{\rm min}} \leq \frac{c_s}{\lambda}\, \rho_{\rm min}^{-s}
\qquad\text{and}\qquad
\frac{1}{\rho_{\rm max}} \geq \frac{c_s}{\lambda}\, \rho_{\rm max}^{-s} \,.
\]
Thus, $\rho_{\rm max}\leq (c_s/\lambda)^{1/(1-s)} \leq \rho_{\rm min}$, which means that $\rho$ is constant. Therefore
the curvature is constant and \(K\) is a disk, as claimed.
\end{proof}


\subsection{Proof of Theorem \ref{prop:planar-convex}}\label{disk-constant}

Let us summarize what we have shown so far in this section. In Subsection \ref{existence-of-a-convex-maximizer} we have shown that the optimization problem \eqref{eq:optproblem} has an optimizer and then in Proposition \ref{eleq} we have derived the Euler--Lagrange equation satisfied by this optimizer. In Proposition \ref{prop:regularity} we have shown that solutions of the Euler--Lagrange equation are $C^2$ with positive curvature and, finally, in Proposition \ref{cor:uniform-chord} we have shown that the only positively curved $C^2$ solutions to this equation are disks. These facts together imply that
\[
\frac{P_s(K)}{P(K)^{2-s}} \leq \frac{P_s(D)}{P(D)^{2-s}}
\qquad\text{for all convex bodies}\ K\subset\R^2 \,.
\]
Here $D$ is a disk, which, by translation and dilation invariance, can be taken to be the unit disk.

It remains to compute the constant $P_s(D)/P(D)^{2-s}$ explicitly. While this is well known, we present a proof that is natural in view of the arguments in this section. At a point of the unit disk, an inward ray making angle \(\alpha\) with
the inward normal has length \(2\cos\alpha\). Therefore, for all $x\in\partial D$,
\[
 H^D_s(x)=\frac2s\int_{-\pi/2}^{\pi/2}(2\cos\alpha)^{-s}\,d\alpha
 =\frac{2^{1-s}\sqrt\pi}{s}
       \frac{\Gamma((1-s)/2)}{\Gamma(1-s/2)}.
\] 
Inserting this into the Euler--Lagrange equation \eqref{eq:eleqsmooth} and recalling the value of $\lambda$ from Proposition~\ref{eleq} gives
\[
\frac{2^{1-s}\sqrt\pi}{s}\frac{\Gamma((1-s)/2)}{\Gamma(1-s/2)} = \frac{(2-s)P_s(D)}{P(D)} \,.
\]
Multiplying both sides by $P(D)^{-1+s} = (2\pi)^{-1+s}$ gives the claimed value of the constant. This completes the proof of Theorem \ref{prop:planar-convex}.


\section{Arbitrary planar sets}\label{sec:planar-general}

Our goal in this section is to remove the convexity assumption from Theorem \ref{prop:planar-convex}. This is achieved in Subsection \ref{sec:reductionconv}. In the final Subsection \ref{sec:additional} we give a simpler proof of the sharp inequality without classifying optimizers.


\subsection{Reduction to the convex case}\label{sec:reductionconv}

We begin by noting that the optimization problem corresponding to Theorem \ref{thm:plane} has an optimizer, that is, the supremum in
\begin{equation}\label{eq:optproblemgen}
    S_s:= \sup \left\{ \frac{P_s(E)}{P(E)^{2-s}} :\ E\subset\R^2 \ \text{measurable with}\ 0<|E|<\infty \,,\ P(E)<\infty \right\}
\end{equation}
is attained. 

This is proved in \cite{DCNRV}. In the present, two-dimensional case there is the following alternative proof, based on arguments that appear in the present paper. First one argues as in Step 1 of the proof of Theorem \ref{thm:reductionconv} below that it suffices to consider indecomposable sets. This allows one to argue as in Subsection \ref{existence-of-a-convex-maximizer}. Indeed, indecomposability leads to (essential) boundedness (see, for example, \cite[Lemma 2.13]{DMNP}, which is deduced from Fleming's theorem, used in Step 2 of the proof of Theorem~\ref{thm:reductionconv}). The use of Blaschke's selection theorem in Subsection \ref{existence-of-a-convex-maximizer} is replaced by \cite[Theorem 12.26]{Maggi}.

Our goal in this subsection will be to prove the following result.

\begin{theorem}\label{thm:reductionconv}
    Let $E\subset\R^2$ be an optimizer for \eqref{eq:optproblemgen}. Then $E$ is convex (up to a set of zero Lebesgue measure). 
\end{theorem}

Clearly, Theorem \ref{thm:reductionconv}, together with Theorem \ref{prop:planar-convex}, implies Theorem \ref{thm:plane}.

\begin{proof}
    \emph{Step 1.} In this step we show that any optimizer for \eqref{eq:optproblemgen} is indecomposable and has no holes.

    We recall that (in any dimension $n$) a set $E\subset\R^n$ of finite perimeter is called \emph{incomposable} if for any measurable partition $(A,B)$ of $E$ with $P(E) = P(A)+P(B)$ at least one of $A$ and $B$ has measure zero. Moreover, a \emph{hole} of an indecomposable set $E$ is an $M$-connected component of $\R^n\setminus E$ with finite, positive measure. We refer to \cite{ACMM01} for a detailed discussion of these notions.

    We assume, by contradiction, $E=A\cup B$ is an optimizer with $|A\cap B|=0$, $P(E)=P(A)+P(B)$ and both $|A|$ and $|B|$ positive. We have
    \[
    P_s(E) = P_s(A) + P_s(B) - 2 \iint_{A\times B} \frac{\dd x \dd y}{|x-y|^{2+s}}
    < P_s(A) + P_s(B) \,,
    \]
    where the inequality is strict since $A$ and $B$ have positive measure. Recalling that $S_s$ denotes the optimal constant (see \eqref{eq:optproblemgen}), we find
    \[
    P_s(E) < P_s(A) + P_s(B) \leq S_s \left( P(A)^{2-s} + P(B)^{2-s} \right) \leq S_s \left( P(A) + P(B) \right)^{2-s} = S_s \, P(E)^{2-s} \,,
    \]
    which contradicts optimality of $E$. Thus, $E$ is indecomposable.

    Next, assume, again by contradiction, that for an optimizer $E$ there is an $M$-connected component $Y$ of $\R^n\setminus E$ with finite measure. Let $E':=E\cup Y$. It follows from \cite[Proposition 9]{ACMM01} and its proof that $P(E)=P(E')+P(Y)$. Moreover, we have
    \[
    P_s(E) = P_s(E') + P_s(Y) - 2 \iint_{Y\times (E')^c} \frac{\dd x \dd y}{|x-y|^{2+s}} < P_s(E') + P_s(Y) \,,
    \]
    where the inequality is strict since $Y$ and $(E')^c$ have positive measure. We find
    \[
    P_s(E) < P_s(E') + P_s(Y) \leq S_s \left( P(E')^{2-s} + P(Y)^{2-s} \right) \leq S_s \left( P(E') + P(Y) \right)^{2-s} = S_s \, P(E)^{2-s} \,,
    \]
    which contradicts optimality of $E$. Thus, $E$ has no holes.

    We emphasize that this step did not use the two-dimensional nature of the problem. This will enter in the next step.

    \medskip

    \emph{Step 2.} Let $E$ be an optimizer for \eqref{eq:optproblemgen}. From Step 1, we know that $E$ is indecomposable and has no holes. By a result of Fleming \cite{F57,F60} (see also \cite[Subsection 4.2.25]{F69}, \cite[Theorem~7]{ACMM01}, \cite[Proposition~2.7]{DMNP}) there is a rectifiable Jordan curve \(\Gamma\) such that \(E\) agrees almost everywhere with its bounded interior and
    \[ P(E)=\mathcal H^1(\Gamma) \,. \]
    From now on, we replace \(E\) by this bounded interior.
    
    With this normalization in place, we want to show that $E$ is convex. Assume, by contradiction, that it is not. Setting $K:=\operatorname{conv}(\overline E)$, the contradiction assumption implies that $\partial K\not\subset\Gamma$. (Indeed, an inclusion of one Jordan curve in another forces equality, which
here would give $E=\operatorname{int}K$.) Choose a connected component
of the relatively open set $\partial K\setminus\Gamma$. Every
extreme point of $K$ belongs to $\Gamma$, so this component is an open
line segment $J=(a,b)$. Its closure $\overline J$ lies in a supporting line $\ell$,
has distinct endpoints on $\Gamma$, and has relative interior
disjoint from $\overline E$.

The points $a$ and $b$ divide the Jordan curve $\Gamma$ into two complementary arcs $\alpha$ and $\beta$, which are disjoint except for their endpoints. Then $J\cup\alpha$ and $J\cup\beta$ are both Jordan curves and, by the Jordan curve theorem, each one separates the exterior into one bounded and one unbounded component. One of the bounded components obtained in this way contains $E$ and one does not. We make the convention that the interior of $J\cup\beta$ contains $E$ and we let $D$ denote the interior of $J\cup\alpha$. Then $D$ is open, bounded, non-empty and disjoint from $E$. In particular, $|D|>0$. Both $E$ and $D$  lie in the supporting half-plane $H_-$ determined by $\ell$ that contains $K$.

Let $\sigma$ be
reflection across $\ell$. The curve
$\Gamma':=\beta\cup\sigma(\alpha)$ is again a rectifiable Jordan
curve. Indeed, the two arcs lie in opposite closed half-planes, so
an intersection must lie on $\ell$. Reflection fixes such a point,
and hence an intersection would already be an intersection of
$\alpha$ with $\beta$; these have only the endpoints $a,b$ in common.
The bounded interior of $\Gamma'$ is,
up to null sets,
\[
 E':=E\cup D\cup\sigma(D) \,.
\]
Reflection preserves arc length. Using the perimeter identity for
rectifiable Jordan curves, we obtain
\[
 P(E')=\mathcal H^1(\beta)+\mathcal H^1(\sigma(\alpha))
      =\mathcal H^1(\Gamma)=P(E) \,.
\]

Put $A:=D\cup\sigma(D)$ and $G:=H_-\setminus D$.
Reflection symmetry gives $A^c=G\mathbin{\dot\cup}\sigma(G)$ (up
to the reflecting line) and, consequently,
\[
P_s(A)=2\iint_{A\times G} \frac{\dd x \dd y}{|x-y|^{2+s}} \,.
\]
Since $E\subset G$ and $|E\cap A|=0$ we have
\[
 P_s(E')-P_s(E)
   =P_s(A)-2\iint_{A\times E} \frac{\dd x \dd y}{|x-y|^{2+s}}
   =2\iint_{A\times (G\setminus E)} \frac{\dd x \dd y}{|x-y|^{2+s}} >0 \,,
\]
where the inequality is strict since $A$ and $G\setminus E$ have positive measure. We find
\[
 P_s(E)<P_s(E')\leq S_s \, P(E')^{2-s}=S_s \, P(E)^{2-s} \,,
\]
which contradicts optimiality of $E$ Thus, $E$ is convex, as claimed.
\end{proof}

We emphasize that, in the context of sets with smooth boundary, the reflection argument in the above proof is one of the well-known ways to reduce the proof of the standard isoperimetric inequality \eqref{eq:isoper} to the convex case. In the above proof we extend this argument to the case of the fractional perimeter. What makes the proof somewhat subtle is our limited knowledge of the regularity of the boundary of optimizers. We overcame this difficulty using Fleming's theorem \cite{F57}. Another option would be to use the regularity theory for almost minimizers of the perimeter; see, e.g., \cite[Part III]{Maggi}.

Finally, we mention that the proof of Theorem \ref{thm:plane} that we have given by combining Theorem \ref{prop:planar-convex} and Theorem \ref{thm:reductionconv} is slightly redundant. Once we have the existence of an optimizer of the optimization problem \eqref{eq:optproblemgen} and the validity of Theorem \ref{thm:reductionconv}, we already know the existence of an optimizer of the optimization problem \eqref{eq:optproblem}, so the material in Subsection \ref{existence-of-a-convex-maximizer} is unnecessary.

We have chosen to include Subsection \ref{existence-of-a-convex-maximizer}, because in the following subsection we will give a proof of the sharp inequality \eqref{eq:planar-main} (without classifying the cases of equality) that requires the existence of an optimizer of \eqref{eq:optproblem}. The advantage of this simpler proof is that it does not rely on Fleming's theorem.


\subsection{Additional remark}\label{sec:additional}

Let us explain how the sharp inequality \eqref{eq:planar-main} follows from the result for convex sets in Theorem~\ref{prop:planar-convex}. The proof uses the same reflection idea as the proof of Theorem \ref{thm:reductionconv}, but is technically simpler, since it only deals with sets with polyhedral boundary. This comes at the expense of not classifying the cases of equality.

\begin{proof}[Proof of \eqref{eq:planar-main}]
    \emph{Step 1.}
    Let $E\subset\R^2$ be a measurable with $0<|E|<\infty$ and $P(E)<\infty$. It follows from \cite[Proposition 13.13]{Maggi} that there is a sequence $(E_j)$ of bounded sets with polyhedral boundary such that $|E_j\Delta E|\to 0$ and $P(E_j)\to P(E)$ as $j\to\infty$. Thus, if we prove the inequality \eqref{eq:planar-main} for $E_j$, it will also hold for $E$. Thus, we may assume that $E$ is bounded with polyhedral boundary.

    \medskip

    \emph{Step 2.}
    We decompose the bounded set \(E\) with polyhedral boundary into finitely many indecomposable components \(E_i\) and, for each $E_i$, we consider its saturation \(F_i\) and its holes \(Y_{ij}\). Then, by similar computations as in Step 1 of the proof of Theorem \ref{thm:reductionconv},
    $$ 
    P_s(E)\le\sum_iP_s(E_i), \qquad P_s(E_i)\le P_s(F_i)+\sum_jP_s(Y_{ij}), 
    $$
    while
    $$ 
    P(E)=\sum_i\left(P(F_i)+\sum_jP(Y_{ij})\right). 
    $$
    Note that each one of the sets $F_i$ and $Y_{ij}$ is bounded by a simple polygon. Thus, if we prove the inequality for each such set, then, in view of the elementary inequality
    $$ 
    \sum_k a_k^{\,2-s}\le\left(\sum_k a_k\right)^{2-s},
    $$
    it will also hold for $E$. Thus, we may assume that $E$ is the interior of a simple polygon. 
    
    \medskip

    \emph{Step 3.}
    For \(E\) the interior of a simple polygon, let \(D\) be one of
its pockets, that is, a maximal connected region interior to the convex hull and exterior to the polygon. Let \(\ell\) be the supporting
line of the convex hull of $E$ that bounds one side of $D$, let \(H_-\) be the open half-plane determined by $\ell$ that contains \(E\), and let \(\sigma\) be the reflection in \(\ell\). Thus
\(D\subset H_-\setminus E\) and \(\sigma(D)\subset \R^2\setminus\overline{H_-} =:H_+\). Up to null sets, the polygon obtained by reflecting the pocket chain has interior
\[
 E':=E\cup D\cup \sigma(D) \,.
\] 
It is still a simple polygon, and its perimeter equals that of \(E\),
because the changed boundary chain is replaced by an isometric
reflection.

Put \(A:=D\cup \sigma(D) \) and \(G:=H_-\setminus D\). By the same computation as in Step 2 of the proof of Theorem \ref{thm:reductionconv} we find
\[
P_s(E')-P_s(E)
   =2\iint_{A\times (G\setminus E)} \frac{\dd x \dd y}{|x-y|^{2+s}} >0 \,,
\]

Thus, the sharp inequality follows for $E$ if we can prove it for $E'$. Of course this observation can be iterated. The Erd\H{o}s--Nagy theorem (see \cite{DGOT} for the history and a complete proof) states that every simple planar polygon becomes
convex after finitely many successive single-pocket reflections. Since Theorem \ref{prop:planar-convex} gives the sharp inequality \eqref{eq:planar-main} for convex sets and, in particular, for convex polygons, we obtain the same inequality for arbitrary simple polygons, as we set out to prove.    
\end{proof}

Finally, we mention that the preceding proof allows one to avoid the existence proof in Subsection~\ref{sec:reductionconv}, at the expense of keeping the existence proof in Subsection \ref{existence-of-a-convex-maximizer}. Indeed, Theorem \ref{prop:planar-convex} (whose proof uses Subsection \ref{existence-of-a-convex-maximizer}) and the preceding proof give the sharp inequality \eqref{eq:planar-main} with equality for disks. To show that these are the only cases of equality, one can directly appeal to Theorem \ref{thm:reductionconv}, without the need of an existence proof. Thus, only one of the existence proofs in Subsections \ref{existence-of-a-convex-maximizer} and \ref{sec:reductionconv} is needed for a complete proof of Theorem \ref{thm:plane} and readers can make a choice according to their taste.


\subsection*{Acknowledgements}

R.L.F. acknowledges partial support through the German Research Foundation through EXC-2111-390814868, TRR 352–Project-ID 470903074, and FR 2664/3-1. P.I. acknowledges partial support from the US NSF CAREER grant DMS-2152401, US NSF grant DMS-2554183, a Simons Fellowship, and a Humboldt Research Fellowship for Experienced Researchers. The authors acknowledge the use of AI tools. All mathematical arguments and proofs in the final manuscript were checked and written by the authors.



\end{document}